\documentclass[preprint,12pt]{elsarticle}
\usepackage{amssymb}
\usepackage{amsthm}
\usepackage{amsmath}
\usepackage{epic}
\newtheorem{thm}{Theorem}[section]

\newtheorem{lem}[thm]{Lemma}

\newtheorem{prou}[thm]{Definition}
\newtheorem{example}[thm]{Example}
\journal{arXiv}

\begin{document}

\begin{frontmatter}

\title{Minimum (maximum) rank of tensors and the sign nonsingular tensors}
\author[label1]{Changjiang Bu}\ead{buchangjiang@hrbeu.edu.cn}
\author[label1]{Wenzhe Wang}
\author[label3]{Lizhu Sun}
\author[label1,label2]{Jiang Zhou}

\address[label1]{College of Science, Harbin Engineering University, Harbin 150001, PR China}
\address[label3]{School of Science, Harbin Institute of Technology, Harbin 150001, PR China}
\address[label2]{College of Computer Science and Technology, Harbin Engineering University, Harbin 150001, PR China}

\begin{abstract}
In this paper, we define the minimum (maximum) rank, term rank and the sign nonsingular of tensors.
The sufficiency and necessity for the minimum rank of a real tensor to be $1$ is given. And we show that the maximum rank of a tensor is not less than the term rank. We also prove that the minimum rank of a sign nonsingular tensor is not less than the dimension of it.
And we get some characterizations of a tensor having sign left or sign right inverses.

\end{abstract}


\end{frontmatter}

\section{Introduction}
\label{Introduction }
For a positive integer $n$, let $[n]=\{1,\ldots,n\}$. Let $\mathbb{R}^{n_1\times\cdots\times n_k}$ be the set of the order $k$ tensors over real field. An order $k$ real tensor $\mathcal{A} =(a_{i_1\cdots i_k})\in \mathbb{R}^{n_1\times\cdots\times n_k}$ is a multidimensional array with $ n_1\times n_2\times\cdots\times n_k$ entries. When $k=2$, $\mathcal{A}$ is an $n_1\times n_2$ matrix. If $n_1=\cdots=n_k=n$, then $\mathcal{A}$ is called an order $k$ dimension $n$ tensor. The order $k$ dimension $n$ tensor $\mathcal{I}=(\delta_{i_1\cdots i_k})$ is called a unit tensor, where $\delta_{i_1\cdots i_k}=1$ if $i_1=\cdots=i_k$, $\delta_{i_1\cdots i_k}=0$ otherwise. There are some results on the research of tensors in [1-7].

 For the nonzero vector $\alpha_j\in\mathbb{R}^{n_j}$ ($j=1,\ldots,k$), let $(\alpha_j)_i$ be the $i$-th component of $\alpha_j$. The \textit{Segre outer product} of $\alpha_1,\ldots,\alpha_k$, denoted by $\alpha_1\otimes\cdots\otimes \alpha_k$, is called the rank one tensor $\mathcal{A} =(a_{i_1\cdots i_k})$ with entries $a_{i_1\cdots i_k}=(\alpha_1)_{i_1}\cdots(\alpha_k)_{i_k}$ (see \cite{ck8}). The rank of a tensor $\mathcal{A}\in\mathbb{R}^{n_1\times \cdots\times n_k}$, denoted by ${\rm rank}(\mathcal{A})$,
is the smallest $r$ such that $\mathcal{A}$ can be written as a sum of $r$ rank one tensors as follows,
\begin{align}
\mathcal{A}=\sum_{j=1}^r\alpha_{1}^j\otimes\cdots\otimes \alpha_{k}^j,\label{s1.1}
\end{align}
where $\alpha_{i}^j\neq0\ \mbox{and}\ \alpha_{i}^j\in\mathbb{R}^{n_i},~i=1,\ldots,k,~j=1,\ldots,r$ (see \cite{ck7,ck8}).

For the vector $x=(x_1,x_2,\ldots ,x_n)^\mathrm{T}$ and an order $k$ dimension $n$ tensor $ \mathcal{A}$, $\mathcal{A}x^{k-1}$ be a dimension $n$ vector whose $i$-th component is
\begin{eqnarray*}
(\mathcal{A}x^{k-1})_i=\sum\limits_{{i_2}, \ldots ,{i_k} \in [n]} {a_{ii_2\cdots i_k}x_{i_2}x_{i_3}\cdots x_{i_k}},
\end{eqnarray*}
where $i\in[n]$ (see \cite{ck4}).

\cite{ck6} defined the \textit{general tensor product}. For the dimension $n$ tensors $\mathcal{A}=(a_{i_1\cdots i_m})$ and $\mathcal{B}=(b_{i_1\cdots i_k})$ ($m\geq 2~,~k\geq1$),
the product of them is an order $(m-1)(k-1)+1$ tensor with entry
\[
(\mathcal{A\cdot B})_{i\alpha_1\cdots \alpha_{m-1}}=\sum_{i_2,\ldots,i_m\in[n]}a_{ii_2\cdots i_m}b_{i_2\alpha_1}\cdots b_{i_m\alpha_{m-1}},
\]
where $i\in[n]$, $\alpha_1,\ldots,\alpha_{m-1}\in[n]^{k-1}$. And if $\mathcal{A}\cdot\mathcal{B}=\mathcal{I}$, then $\mathcal{A}$ is called an order $m$ left inverse of $\mathcal{B}$ and $\mathcal{B}$ is called an order $k$ right inverse of $\mathcal{A}$ (see \cite{ck11}). The determinant of an order $k$ dimension $n$ tensor $\mathcal{A}$, denoted by ${\rm det}(\mathcal{A})$, is the resultant of the system of homogeneous equation $\mathcal{A}x^{k-1}=0$, where $x\in\mathbb{R}^n$ (see \cite{ck5}). \cite{ck4} researched the determinant of symmetric tensors. \cite{ck6} proved that ${\rm det}(\mathcal{A})$ is the unique polynomial on the entries of $\mathcal{A}$ satisfying the following three conditions:

(1) ${\rm det}(\mathcal{A})=0$ if and only if the system of homogeneous equation $\mathcal{A}x^{k-1}=0$ has a nonzero solution;

(2) ${\rm det}(\mathcal{I})=1$;

(3) ${\rm det}(\mathcal{A})$ is an irreducible polynomial on the entries of $\mathcal{A}$ when the entries $a_{i_1\cdots i_k}$ $(i_1,\ldots,i_k\in[n])$ of $\mathcal{A}$ are all viewed as independent different variables. If ${\rm det}(\mathcal{A})\neq0$, then $\mathcal{A}$ is called nonsingular tensor.

The sign pattern of matrices is proposed by P.A. Samuelson in the problem of sign-solvable linear systems (see \cite{ck1}).
For a real number $a$, let ${\rm sgn}(a)$ be the sign of $a$, ${\mathop{\rm sgn}} (a) = \left\{ {\begin{array}{*{20}c}
   {0\,,\;\;\; a = 0}  \\
   { -1 ,\;a < 0} \\
   { 1 ,\;\;\; a > 0}  \\
\end{array}} \right.$. For a real matrix $A=(a_{ij})_{m\times n}$, let ${\rm sgn}(A)=({\rm sgn}(a_{ij}))_{m\times n}$ denote the sign pattern of $A$ and let $\mathcal{Q}(A)= \{ \hat{A}|~{\mathop{\rm sgn}} (\hat{A}) = {\mathop{\rm sgn}} (A)\}$ denote the sign pattern class (or qualitative class) of $A$ (see \cite{ck29}). If each matrix in $\mathcal{Q}(A)$ has full row rank, then $A$ is called the \textit{L}-matrix. When $m=n$ and $A$ is an \textit{L}-matrix, $A$ is called a sign nonsingular matrix (abbreviated $SNS$ matrix). $SNS$ matrices play an important role in the research of sign-solvability linear systems. For a square matrix $A$,
the solution of the linear system $Ax=0$ has a unique sign pattern if and only if $A$ is an $SNS$ matrix (see \cite{ck12}). \cite{ck13,ck25} characterised the matrices with sign M-P inverse. \cite{ck14} researched the matrices with sign Drazin inverse. \cite{ck26,ck27} researched the sign pattern of matrices allowing nonnegative \{1, 3\} inverse, M-P inverse, and left inverse.

For a real matrix $A$, let ${\rm mr}(A)={\rm min}\{{\rm rank}(B)|B\in\mathcal{Q}(A)\}$ and ${\rm Mr}(A)={\rm max}\{{\rm rank}(B)|B\in\mathcal{Q}(A)\}$ be the minimum rank and maximum rank of $A$, respectively. The term rank of $A$, denoted by $\rho(A)$, is the maximum number of nonzero entries of $A$ no two of which are in the same row or same column \cite{ck15}. It is well known that ${\rm Mr}(A)=\rho(A)$ (see \cite{ck15}). The researches on the minimum rank of matrices have many important applications in communication complexity and neural networks [20-22].
 Some results of the matrix minimum rank are given in \cite{ck16,ck17}.

In this paper, we will research the sign pattern of real tensors.
Next, we give the definitions of the sign pattern, minimum (maximum) rank, term rank and sign nonsingular of tensors.

\begin{prou}
Let $\mathcal{A} =(a_{i_1\cdots i_k})\in  \mathbb{R}^{n_1\times\cdots\times n_k}$. The tensor ${\rm sgn}(\mathcal{A})=({\rm sgn}(a_{i_1\cdots i_k}) )$ is called the sign pattern of $\mathcal{A}$ and $\mathcal{Q}(\mathcal{A}) = \{ \hat{\mathcal{A}}|~{\mathop{\rm sgn}} (\hat{\mathcal{A}}) = {\mathop{\rm sgn}} (\mathcal{A})\}$ is called the sign pattern class (or qualitative class) of $\mathcal{A}$.
 \end{prou}

 \begin{prou}\label{defi 1}
Let $\mathcal{A}$ be a real tensor. The ${\rm mr}(\mathcal{A})={\rm min}\{{\rm rank}(\mathcal{B})|\mathcal{B}\in\mathcal{Q}(\mathcal{A})\}$ is called the the minimum rank of $\mathcal{A}$;
 ${\rm Mr}(\mathcal{A})={\rm max}\{{\rm rank}(\mathcal{B})|\mathcal{B}\in\mathcal{Q}(\mathcal{A})\}$ is called the maximum rank of $\mathcal{A}$; the maximum number of nonzero entries of $\mathcal{A}$ no two of which have the same index in the same dimension is called the term rank of $\mathcal{A}$, denoted by $\rho(\mathcal{A})$.
 \end{prou}
For a dimension $n$ unit tensor $\mathcal{I}$, it is easy to see $\rho(\mathcal{I})=n$.

\begin{prou}
Let $\mathcal{A}$ be an order $k$ dimension $n$ tensor. If each tensor in $\mathcal{Q}(\mathcal{A})$ is a nonsingular tensor, then $\mathcal{A}$ is called a sign nonsingular tensor (abbreviated $SNS$ tensor).
 \end{prou}

Let $\mathcal{A}$ be an order $k$ dimension $n$ tensor. If the sign pattern of the solutions of $\widetilde{\mathcal{A}}x^{k-1}=0$ are the same
for all the tensors $\widetilde{\mathcal{A}}\in\mathcal{Q}(\mathcal{A})$, then $0$ is the unique solution of $\widetilde{\mathcal{A}}x^{k-1}=0$.
Therefore, ${\rm det}(\widetilde{\mathcal{A}})\neq 0$, so $\mathcal{A}$ is a sign nonsingular tensor. Thus, we have the equation ${\mathcal{A}}x^{k-1}=0$ is sign solvable if and only if $\mathcal{A}$ is a sign nonsingular tensor.

\begin{prou}
Let $\mathcal{A}$ be an order $k$ dimension $n$ tensor. If each tensor in $\mathcal{Q}(\mathcal{A})$ has an order $m$ left (right) inverse, then $\mathcal{A}$ is called having an order $m$ sign left (right) inverse.
\end{prou}

We organize this paper as follows. In the section 2, some lemmas are presented. In the section 3, we give that
the sufficient condition for real tensors having the same minimum rank or maximum rank; the relations of the minimum (maximum) rank between tensor and subtensor;
the sufficiency and necessity for the minimum rank of a real tensor to be $1$. And in the section 4, some results are showed including that the minimum rank of a sign nonsingular tensor is not less than its dimension; the maximum rank of tensor is not less than the term rank; the tensor having sign left (right) inverse is a sign nonsingular tensor; the sufficiency and necessity for a real tensor having order $2$ sign left or right inverses.

\section{Preliminaries}

Let $\mathcal{A}=(a_{i_1\cdots i_k})\in\mathbb{R}^{n_1\times\cdots\times n_k}$ and the matrix $B^{(p)}=(b_{ji}^{(p)})\in\mathbb{R}^{c_p\times n_p}~(p=1,\ldots,k)$. The multilinear transform of $\mathcal{A}$ is defined as follows
 \[
 \mathcal{A}'=(a'_{j_1\cdots j_k})=(B^{(1)},\ldots,B^{(k)})\cdot\mathcal{A}\in\mathbb{R}^{c_1\times\cdots\times c_k},
 \]
where
$
 a'_{j_1\cdots j_k}=\sum_{i_1,\ldots,i_k=1}^{n_1,\ldots,n_k}b^{(1)}_{j_1i_1}\cdots b^{(k)}_{j_ki_k}a_{i_1\cdots i_k}$ (see \cite{ck7,ck8}).

\begin{lem}\textup{\cite{ck8}}\label{le2.1}
  Let $\mathcal{A}\in\mathbb{R}^{n_1\times\cdots\times n_k}$ and $L_i\in\mathbb{R}^{c_i\times n_i}$ $(i=1,\ldots,k)$, then
  \[
  {\rm rank}((L_1,\ldots,L_k)\cdot\mathcal{A})\leq {\rm rank}(\mathcal{A}).
  \]
If $L_1, \ldots, L_k$ are nonsingular, then the equality holds.
\end{lem}

For a tensor $\mathcal{A}=(a_{i_1\cdots i_k})\in\mathbb{R}^{n_1\times \cdots\times n_k}$,
let the vector
\begin{align*}
 &\mathcal{A}_{d_1\cdots d_{s-1}\bullet d_{s+1}\cdots d_k}\\
 =&(a_{d_1\cdots d_{s-1} 1 d_{s+1}\cdots d_k},~a_{d_1\cdots d_{s-1} 2 d_{s+1}\cdots d_k},~\cdots~,~a_{d_1\cdots d_{s-1} n_s d_{s+1}\cdots d_k})^\top\in\mathbb{R}^{n_s}
\end{align*}
 and
$$V_s(\mathcal{A})=\{\mathcal{A}_{d_1\cdots d_{s-1}\bullet d_{s+1}\cdots d_k}|~d_j\in[n_j],~j=1,\ldots,s-1,s+1,\ldots,k\}.$$
Let
$$r_s(\mathcal{A})={\rm dim}({\rm span}(V_s(\mathcal{A})))$$
 be the $s$-th order {\rm rank} of $\mathcal{A}$. The multilinear rank of $\mathcal{A}$ is denoted by $$rank_{\boxplus}(\mathcal{A})=(r_1(\mathcal{A}),\ldots , r_k(\mathcal{A}))~ (see~[8]).$$

 Let $a=(a_1,a_2,\ldots,a_n)$ and $b=(b_1,b_2,\ldots,b_n)$ be tow real tuples. In this paper, $a\leq b$~(or $a=b$) implies $a_i\leq b_i$ (or $a_i=b_i$), $i=1,2,\ldots,n $.

\begin{lem}\textup{\cite{ck8}}\label{le2.2}
  Let $\mathcal{A}\in\mathbb{R}^{n_1\times\cdots\times n_k}$
and $L_i\in\mathbb{R}^{c_i\times n_i}$ $(i=1,\ldots,k)$, then
  \[
  {\rm rank}_\boxplus((L_1,\ldots,L_k)\cdot\mathcal{A})\leq {\rm rank}_\boxplus(\mathcal{A}).
  \]
If $L_1,\ldots, L_k$ are nonsingular, then the equality holds.
 \end{lem}

 \begin{lem}\textup{\cite{ck8}}\label{le2.3}
 Let $\mathcal{A}\in\mathbb{R}^{n_1\times\cdots\times n_k}$, then
 \[
 {\rm rank}(\mathcal{A})\geq{\rm max}\{r_i(\mathcal{A})|~i=1,\ldots,k\}.
 \]
 \end{lem}

\begin{lem}\textup{\cite{ck11}}\label{le2.4}
Let $f_1(x_1,\ldots,x_n),~\ldots~,f_r(x_1,\ldots,x_n)$ be the homogeneous polynomials with degree $m$ and let $$F(x_1,\ldots,x_n) = (f_1(x_1,\ldots,x_n),\ldots, f_r(x_1,\ldots,x_n))^\top.$$
If $r < n$, then $F(x_1,\ldots,x_n) = 0$ has a nonzero solution.
\end{lem}

\begin{lem}\textup{\cite{ck5}}\label{le2.5}
Let $\mathcal{A}$ and $\mathcal{B}$ be the dimension $n$ tensors of order $m$ and order $k$, respectively, then ${\rm det}(\mathcal{A\cdot B})=\left({\rm det}(\mathcal{A})\right)^{(k-1)^{n-1}}\left({\rm det}(\mathcal{B})\right)^{(m-1)^n}$.
\end{lem}

\begin{lem}\textup{\cite{ck11}}\label{le2.6}
For an order $k$ tensor $\mathcal{A}\in\mathbb{R}^{n\times\cdots\times n}$, $\mathcal{A}$ has an order $2$ left inverse if and only if there exists an invertible matrix $P$ such that $\mathcal{A}=P\cdot \mathcal{I}$; $\mathcal{A}$ has an order $2$ right inverse if and only if there exists an invertible matrix $Q$ such that $\mathcal{A}=\mathcal{I}\cdot Q$.
\end{lem}

\section{ Minimum rank and maximum rank of tensors}
For $\mathcal{A}=(a_{i_1\cdots i_k})\in\mathbb{R}^{n_1\times\cdots\times n_k}$, let $\mathcal{A}^{\top(p,q)}=(a'_{i_1\cdots i_{p-1}i_qi_{p+1}\cdots i_{q-1}i_pi_{q+1}\cdots i_k})$ $(p,~q\in[k])$ be the $(p,q)$ transpose of $\mathcal{A}$, where $a'_{i_1\cdots i_{p-1}i_qi_{p+1}\cdots i_{q-1}i_pi_{q+1}\cdots i_k}=a_{i_1\cdots i_p\cdots i_q\cdots i_k}~(i_j\in[n_j],~j=1,\ldots,k)$ (see \cite{ck5}).

The tensor $\mathcal{A}=(a_{i_1\cdots i_k})$ can be written as
\[
\mathcal{A}=\sum_{i_k=1}^{n_k}\sum_{i_{k-1}}^{n_{k-1}}\cdots\sum_{i_1=1}^{n_1}a_{i_1\cdots i_k}e_{i_1}^{(n_1)}\otimes\cdots\otimes e_{i_k}^{(n_k)},
\]
where $e_{i_j}^{(n_j)}$ is the dimension $n_j$ unit vector with $i_j$-th component being $1$, $i_j\in[n_j]$ ($j=1,\ldots,k$). Let $L_i\in\mathbb{R}^{c_i\times n_i}$ $(i=1,\ldots,k)$. Then
\begin{align}
\mathcal{B}=(L_1,\ldots,L_k)\cdot\mathcal{A}=\sum_{i_k=1}^{n_k}\sum_{i_{k-1}}^{n_{k-1}}\cdots\sum_{i_1=1}^{n_1}L_1e_{i_1}^{(n_1)}\otimes\cdots\otimes L_ke_{i_k}^{(n_k)}~(see~[8]).\label{xinequ1}
\end{align}
If $L_i$($i=1,2,\ldots,k$) is a permutation matrix or a diagonal matrix with diagonal entries $1$ or $-1$, by the above formula, it is clear that $(L_1,\ldots,L_k)\cdot\widetilde{\mathcal{A}}\in\mathcal{Q}(\mathcal{B})$ for each $\widetilde{\mathcal{A}}\in\mathcal{Q}(\mathcal{A})$. And the following result can be obtained.

\begin{thm}\label{thm3.1}
For an order $k$ real tensors $\mathcal{A}$. Then ${\rm mr}(\mathcal{A})={\rm mr}(\mathcal{B})$ and ${\rm Mr}(\mathcal{A})={\rm Mr}(\mathcal{B})$, if one of the following holds:

{\rm(1)}
${\rm sgn}(\mathcal{B})={\rm sgn}(\mathcal{A}^{\top(p,q)})$;

{\rm(2)} There exist some permutations matrices $P_1,\ldots,P_k$ such that ${\rm sgn}(\mathcal{B})=(P_1,\ldots,\\P_k)\cdot{\rm sgn}(\mathcal{A})$;

{\rm(3)} There exist some diagonal matrices $D_1,\ldots, D_k$ with diagonal elements $1$ or $-1$ such that ${\rm sgn}(\mathcal{B})=(D_1,\ldots,D_k)\cdot{\rm sgn}(\mathcal{A})$.

\end{thm}
\begin{proof}
Making $(p,q)$ transpose to $\mathcal{A}$ is exchanging the vectors $\alpha_{p}^j$ and $\alpha_{q}^j$ in the (\ref{s1.1}).
It is easy to see that ${\rm rank}(\mathcal{A})={\rm rank}(\mathcal{A}^{\top(p,q)})$, ${\rm mr}(\mathcal{A})={\rm mr}(\mathcal{A}^{\top(p,q)})$ and ${\rm Mr}(\mathcal{A})={\rm Mr}(\mathcal{A}^{\top(p,q)})$. It yields that (1) holds.

Next we prove that (2) and (3) hold. Suppose that $\mathcal{B}=(P_1,\ldots,P_k)\cdot\mathcal{A}$, ${\rm mr}(\mathcal{A})=r_1$ and ${\rm mr}(\mathcal{B})=r_2$, then there exists a tensor $\widetilde{\mathcal{A}}\in\mathcal{Q}(\mathcal{A})$ such that ${\rm rank}(\widetilde{\mathcal{A}})=r_1$. According to (\ref{xinequ1}) and Lemma \ref{le2.1}, we have that $r_1={\rm rank}((P_1,\ldots,P_k)\cdot\widetilde{\mathcal{A}})$. Since $(P_1,\ldots,P_k)\cdot\widetilde{\mathcal{A}}\in\mathcal{Q}(\mathcal{B})$, we get ${\rm mr}(\mathcal{B})=r_2\leq r_1$. Similarly, we can get $r_2\geq r_1$. Therefore, $r_1=r_2$. That is ${\rm mr}(\mathcal{A})={\rm mr}(\mathcal{B})$. By the same method, it yields that ${\rm Mr}(\mathcal{A})={\rm Mr}(\mathcal{B})$.

\end{proof}

For the tensor $\mathcal{A}=(a_{i_1\cdots i_k})\in\mathbb{R}^{n_1\times\cdots\times n_k}$. Let
$$(\mathcal{A})_i^{(j)}=(a_{i_1\cdots i_{j-1}ii_{j+1}\cdots i_k})\in\mathbb{R}^{n_1\times\cdots\times n_{j-1}\times n_{j+1}\times\cdots\times n_k}$$
 is the subtensor of $\mathcal{A}$ obtained by fixing the $j$-th index to be $i$, where $j\in[k]$, $i\in[n_j]$. And it can be regarded as a slice of $\mathcal{A}$ by $j$-th order.
Unfolding $\mathcal{A}$ into the subtensors (slices) by $j$-th order gives that $\mathcal{A}=((\mathcal{A})_1^{(j)},\ldots,(\mathcal{A})_{n_j}^{(j)})$ (see \cite{ck7}).
Obviously, $(\mathcal{A}^{\top(1,p)})^{(1)}_i=(\mathcal{A})^{(p)}_i$ $(i=1,\ldots,n_p, ~p\in\{2,\ldots,k\}$).

\begin{thm}\label{thm3.2}
For the tensor $\mathcal{A} \in \mathbb{R}^{n_1\times \cdots \times n_k}$, let $\mathcal{A}=((\mathcal{A})_1^{(1)},\ldots,(\mathcal{A})_{n_1}^{(1)})$ be the $1$-th unfolded expression and $\mathcal{A}_1=((\mathcal{A})_1^{(1)},\ldots,(\mathcal{A})_{n_1-1}^{(1)})$.

{\rm (1)} If $(\mathcal{A})_{n_1}^{(1)}=0$, then ${\rm mr}(\mathcal{A})={\rm mr}(\mathcal{A}_1)$ and ${\rm Mr}(\mathcal{A})={\rm Mr}(\mathcal{A}_1)$;

{\rm (2)} If ${\rm sgn}((\mathcal{A})_{n_1}^{(1)})=c~{\rm sgn}((\mathcal{A})_{l}^{(1)})$, where $l\in [n_1-1]$, $c=1$ or $-1$, then ${\rm mr}(\mathcal{A})={\rm mr}(\mathcal{A}_1)$.
\end{thm}
\begin{proof}
(1) Let $r={\rm rank}(\mathcal{A})$. According to (\ref{s1.1}),
\[\mathcal{A}=\sum_{j=1}^r\alpha_{1}^j\otimes\cdots\otimes\alpha_{k}^j.\]
Let $\alpha_{1}^j=\left( {\begin{array}{*{20}c}
   \beta_{1}^j  \\
   (\alpha_{1}^j)_{n_1}   \\
\end{array}} \right),\ $ where $\beta_{1}^j\in \mathbb{R}^{n_1-1}$ and $(\alpha_{1}^j)_{n_1} $ is the $n_1$-th component of $\alpha_{1}^j$, $j=1,\ldots,r$. Since $(\mathcal{A})_{n_1}^{(1)}=0$, it yields that
\[
\mathcal{A}=\sum_{j=1}^r\left( {\begin{array}{*{20}c}
   \beta_{1}^j  \\
   0   \\
\end{array}} \right)\otimes\cdots\otimes\alpha_{k}^j.
\]
Therefore, $\mathcal{A}_1=\sum_{j=1}^r
   \beta_{1}^j\otimes\cdots\otimes\alpha_{k}^j$, so ${\rm rank}(\mathcal{A})={\rm rank}(\mathcal{A}_1)$, ${\rm mr}(\mathcal{A})={\rm mr}(\mathcal{A}_1)$ and ${\rm Mr}(\mathcal{A})={\rm Mr}(\mathcal{A}_1)$.
Thus, we get (1) holds.

(2) Without loss of generality, suppose that $l=1$. Let vector
\[
\gamma^j=((\alpha_{1}^j)_1,\ldots,(\alpha_{1}^j)_{n_1-1},c(\alpha_{1}^j)_1)^\top\ (j=1,\ldots,r, c=1~or~-1),\]
 where $(\alpha_{1}^j)_k$ ($k\in[n_1-1]$) is the $k$-th component of $\alpha_{1}^j$. Let
\begin{align*}
\mathcal{B}&=\sum_{j=1}^r\gamma^j \otimes\alpha_{2}^j\otimes\cdots\otimes\alpha_{k}^j
=((\mathcal{A})_1^{(1)},\ldots,(\mathcal{A})_{n_1-1}^{(1)},
c(\mathcal{A})_1^{(1)}).
\end{align*}
Since ${\rm sgn}((\mathcal{A})_{n_1}^{(1)})=c~{\rm sgn}((\mathcal{A})_{1}^{(1)})$, we obtain $\mathcal{B}\in\mathcal{Q}(\mathcal{A})$, so ${\rm mr}(\mathcal{B})={\rm mr}(\mathcal{A})$.
Take matrix $P=\left( {\begin{array}{*{20}c}
   1 & 0 & \cdots & 0  \\
   0 & 1 & \cdots & 0 \\
   \vdots & \vdots & \ddots & \vdots\\
   -c & 0 & \cdots & 1
\end{array}} \right)\in\mathbb{R}^{n_1\times n_1}$. Then we get
\[
P\gamma^j=((\alpha_{1}^j)_1,\ldots,(\alpha_{1})^j_{n_1-1},0)^\top
\]
for each $j\in[r]$. Let $\mathcal{C}=(P,I,\ldots,I)\cdot\mathcal{B}$, then
\[
\mathcal{C}=((\mathcal{A})_1^{(1)},\ldots,(\mathcal{A})_{n_1-1}^{(1)},0).
\]
By (1) of Theorem \ref{thm3.2}, it yields that ${\rm rank}(\mathcal{C})={\rm rank}(\mathcal{A}_1)$ and ${\rm mr}(\mathcal{A})={\rm mr}(\mathcal{A}_1)$.
So it follows from Lemma \ref{le2.1} that ${\rm rank}(\mathcal{B})={\rm rank}(\mathcal{C})={\rm rank}(\mathcal{A}_1)$.
Hence, ${\rm mr}(\mathcal{A}_1)={\rm mr}(\mathcal{B})={\rm mr}(\mathcal{A})$. Thus, we get (2) holds.
\end{proof}

By deleting the rows (columns) whose sign patterns are zero or the same (opposite) to other rows (columns), \cite{ck15} give the conception of the sign condensed matrix and use it to characterise the minimum rank of a matrix. Similarly, we give the conception of the condensed tensor and use it to characterise the minimum rank of a tensor.

We condense a tensor $\mathcal{A}\in \mathbb{R}^{n_1\times \cdots\times n_k}$ by deleting the slices whose sign patterns are zero or the same (opposite) to other slices in $i$-th order unfolded expression $\mathcal{A}=((\mathcal{A})_1^{(i)},\ldots,(\mathcal{A})_{n_i}^{(i)})$ ($i=1,\cdots,k$). The condensed tensor of $\mathcal{A}$ is gotten by doing this deletion form $1$-th to $k$-th order unfolded expression of $\mathcal{A}$, denoted by $\mathcal{C}(\mathcal{A})$.

From Theorem \ref{thm3.2}, we can see that the minimum rank of a tensor is unchanged after doing this deletion.
Obviously, the minimum {\rm rank} of $\mathcal{C}(\mathcal{A})$ and $\mathcal{A}$ are the same, that is ${\rm mr}(\mathcal{A})={\rm mr}(\mathcal{C}(\mathcal{A}))$.

There are some results on the sufficient and necessary conditions for the minimum rank of matrices to be $1$ or $2$ (see [15, 16]). Next, we give some results on the tensor whose minimum rank is $1$.

\begin{thm}\label{thm3.6}
Let $\mathcal{A}\in\mathbb{R}^{n_1\times\cdots\times n_k}$. Then ${\rm mr}(\mathcal{A})=1$ if and only if ${\rm sgn}(\mathcal{C}(\mathcal{A}))=+\ \mbox{or}\ -.$
\end{thm}
\begin{proof}
According to above discussion, the sufficiency is obvious.

Next we prove the necessity. Since that ${\rm mr}(\mathcal{A})=1$, then there exists a tensor $\mathcal{A}_1\in\mathcal{Q}(\mathcal{A})$ such that ${\rm rank}{(\mathcal{A}}_1)=1$. Thus,
$\mathcal{A}_1$ can be written as the rank one form as $\mathcal{A}_1=\alpha_1\otimes\cdots\otimes\alpha_k$, where $\alpha_i\neq0$ and $\alpha_i\in\mathbb{R}^{n_i}$ $(i=1,\ldots,k)$.
Deleting the slices of $\mathcal{A}_1$ whose sign patterns are zero or the same (opposite) to other slices, by the $1$-th order (considering the $1$-th vector $\alpha_1$ ), it yields that
\[
\mathcal{C}(\mathcal{A}_1)=\mathcal{C}(\alpha_1\otimes\alpha_2\otimes\cdots\otimes\alpha_k)=\mathcal{C}(a_1\alpha_2\otimes\cdots\otimes\alpha_k),
\]
where $a_1$ is the nonzero element of $\alpha_1$. Next, doing the deletion by the $2$-th order (considering the $2$-th vector $\alpha_2$ ), we get
\[
\mathcal{C}(\mathcal{A}_1)=\mathcal{C}(a_1a_2\alpha_3\otimes\cdots\otimes\alpha_k),
\]
where $a_2$ is the nonzero element of $\alpha_2$.  Doing the deletion by all the other orders, we obtain
\[
\mathcal{C}(\mathcal{A}_1)=a_1a_2\cdots a_k\neq0,
\]
where $a_i$ is the nonzero element of $\alpha_i$ $(i=3,\ldots,k)$. Obviously, ${\rm sgn}(\mathcal{C}(\mathcal{A})_1)=+\ \mbox{or}\ -$. Note that $\mathcal{A}_1\in \mathcal{Q}(\mathcal{A})$,
then ${\rm sgn(\mathcal{C}(\mathcal{A})_1)}={\rm sgn(\mathcal{C}(\mathcal{A}))}=+\ \mbox{or}\ -$.  Hence, the necessity holds.
\end{proof}

For $\mathcal{A}\in \mathbb{R}^{n_1\times\ldots\times n_k}$, let $\gamma_i$ be a non-empty subset of the set $[n_i]$ and $\mathcal{A}[\gamma_1,\ldots,\gamma_k]$ be a subtensor of $\mathcal{A}$ obtained by deleting the elements whose $i$-th indices are not in $\gamma_i$, $i=1,\ldots,k$, then $\mathcal{A}[\gamma_1,\ldots,\gamma_k]=(b_{i_1\cdots i_k})\in\mathbb{R}^{m_1\times\cdots\times m_k}$, where $m_i$ denotes the number of element of set $\gamma_i$, $i=1,\ldots,k$.

In the following theorem, the relation of the minimum (maximum) rank between a tensor and its subtensors is given.

\begin{thm}\label{thm3.7}
For $\mathcal{A}\in\mathbb{R}^{n_1\times\cdots\times n_k}$, let $\gamma_j$ be a subset of the set $[n_j]$ $(j=1,\ldots,k)$ and $\mathcal{B}=\mathcal{A}[\gamma_1,\ldots,\gamma_k]$ is a subtensor of $\mathcal{A}$, then ${\rm mr}(\mathcal{A})\geq{\rm mr}(\mathcal{B})$ and ${\rm Mr}(\mathcal{A})\geq{\rm Mr}(\mathcal{B})$.
\end{thm}
\begin{proof}
Without loss of generality, suppose that $\mathcal{A}\neq 0$. Let ${\rm mr}(\mathcal{A})=r_1$ and ${\rm mr}(\mathcal{B})=r_2$, then there exist a tensor $\widetilde{\mathcal{A}}\in\mathcal{Q}(\mathcal{A})$ and nonzero vectors $\alpha_{i}^j\in\mathbb{R}^{n_i}$ $(i=1,\ldots,k,~j=1,\ldots,r_1)$ such that
\[
\widetilde{\mathcal{A}}=\sum_{j=1}^{r_1}\alpha_{1}^j\otimes\cdots\otimes\alpha_{k}^j.
\]
Let $\widetilde{\mathcal{B}}=\sum_{j=1}^{r_1}\alpha_{1}^j[\gamma_1]\otimes\cdots\otimes\alpha_{k}^j[\gamma_k]$. Obviously, $\widetilde{\mathcal{B}}\in\mathcal{Q}(\mathcal{B})$, so $r_1\geq r_2$.

Let ${\rm Mr}(\mathcal{B})=c_1$. Then there exists a tensor $\widetilde{\mathcal{B}}\in\mathcal{Q}(\mathcal{B })$ such that ${\rm rank}(\widetilde{\mathcal{B}})=c_1$.
Take a tensor $\widetilde{\mathcal{A}}\in\mathcal{Q}(\mathcal{A})$ satisfying $\widetilde{\mathcal{A}}[\gamma_1,\ldots,\gamma_k]=\widetilde{\mathcal{B}}$. Suppose that ${\rm rank}(\widetilde{\mathcal{A}})=c_2$, then there exist nonzero vectors $\alpha_{i}^j$ $(i=1,\ldots,k,~j=1,\ldots,c_2)$ such that
\[
\widetilde{\mathcal{A}}=\sum_{j=1}^{c_2}\alpha_{1}^j\otimes\cdots\otimes\alpha_{k}^j.
\]
It is easy to see that
\[
\widetilde{\mathcal{B}}=\sum_{j=1}^{c_2}\alpha_{1_j}[\gamma_1]\otimes\cdots\otimes\alpha_{k_j}[\gamma_k],
\]
so $c_2\geq c_1$. Therefore, ${\rm Mr}(\mathcal{A})\geq c_2\geq c_1={\rm Mr}(\mathcal{B})$. Thus, the theorem holds.
\end{proof}

\section{Sign  nonsingular tensor}
Let $\mathcal{A}$ be the order $k$ tensor as in (\ref{s1.1}). The vector in $V_i(\mathcal{A})$ ($i\in[k]$) is the linear combinations of $\alpha_{i}^1,\ldots,\alpha_{i}^r$, so
$$
r_i(\mathcal{A})={\rm rank}(M_i)\leq\min\{n_1, r\}
,$$ where $M_i=(\alpha_{i}^1~\cdots~\alpha_{i}^r)\in\mathbb{R}^{n_i\times r}$.
If ${\rm mr}_{i}(\mathcal{A})<n_i$, then there exists an invertible matrix $P_i$ such that $P_iM_i$ has at least one zero row.

Let $\mathcal{B}=(I,\ldots, I, P_i, I,\ldots,I)\cdot\mathcal{A}$. Unfolding $\mathcal{B}$ into slices by $i$-th order gives that

\begin{align}\label{equ 4}
 {\cal B} &= (I, \ldots ,I,P_i ,I, \ldots ,I)\cdot{\cal A} \notag \\
  &= \sum_{j=1}^r I \alpha _1^j  \otimes  \cdots  \otimes I\alpha _{i - 1}^j  \otimes P_i \alpha _i^j  \otimes I\alpha _{i + 1}^j  \otimes  \cdots  \otimes I\alpha _k^j  \\
 &= (({\cal B})_1^{(i)} , \ldots ,({\cal B})_{n_i  - 1}^{(i)} ,0).\notag
\end{align}

 Let $\mathcal{T}=(Q_1,\ldots,Q_{i-1},P_i,Q_{i+1},\ldots,Q_k)\cdot\mathcal{A}$, where $Q_j$ is an invertible matrix ($j\in[k]$, $j\neq i$). Unfolding $\mathcal{T}$ into slices by $i$-th order gives that
\begin{eqnarray}\label{equ 5}
\mathcal{T}=(Q_1,\ldots,Q_{i-1},P_i,Q_{i+1},\ldots,Q_k)\cdot\mathcal{A}=((\mathcal{T})_1^{(i)},\ldots,(\mathcal{T})_{n_i-1}^{(i)},0).
\end{eqnarray}

If $A\in \mathbb{R}^{n\times n}$ is a sign nonsingular matrix, then ${\rm mr}(A)={\rm Mr}(A)=n$. For a tensor $\mathcal{A}$, the following result can be gotten.

\begin{thm}\label{thm3.14}
For an order $k$ tensor $\mathcal{A}\in\mathbb{R}^{n\times\cdots\times n}$. If $\mathcal{A}$ is a sign nonsingular tensor, then ${\rm rank}_\boxplus(\mathcal{\widetilde{A}})=(n,\ldots,n)$ for each $\widetilde{\mathcal{A}}\in \mathcal{Q}(\mathcal{A})$ and ${\rm mr}(\mathcal{A})\geq n$.
\end{thm}
\begin{proof}
We prove this theorem by using contradiction. Suppose there exists a tensor $\mathcal{A}_1\in\mathcal{Q}(\mathcal{A})$ such that ${\rm rank}_\boxplus(\mathcal{A}_1)\neq(n,\ldots,n)$, then there exists an integer $i\in[k]$ such that ${\rm r}_i(\mathcal{A}_1)<n$.

If $i=1$, that is ${\rm r}_1(\mathcal{A}_1)<n$. From (\ref{equ 4}), it yields that there exists an invertible matrix $P$ such that
\[
\mathcal{B}=(P,I,\ldots, I)\cdot \mathcal{A}_1=P\cdot\mathcal{A}_1=((\mathcal{B})_1^{(1)},\ldots,(\mathcal{B})_{n-1}^{(1)},0).
\]
 It follows from Lemma \ref{le2.4} that the equation $\mathcal{B}x^{k-1}=0$ has a nonzero solution, where $x\in\mathbb{C}^n$. So ${\rm det}(\mathcal{B})=0$. By Lemma \ref{le2.5}, we get ${\rm det}(\mathcal{B})=({\rm det}(P))^{(k-1)^{n}}{\rm det}(\mathcal{A}_1)$, so ${\rm det}(\mathcal{A}_1)=0$. It is contradict to that $\mathcal{A}$ is a sign nonsingular tensor. Therefore, ${\rm r}_1(\widetilde{\mathcal{A}})=n$ for each $\widetilde{\mathcal{A}}\in \mathcal{Q}(\mathcal{A})$.

If $i\neq1$, without loss of generality, suppose that $i=2$. By above discussion and (\ref{equ 5}), there exists an invertible matrix $P$ such that
\begin{eqnarray}\label{equ 6}
\mathcal{B}=(I, P^\mathrm{T}, \ldots, P^\mathrm{T})\cdot \mathcal{A}_1=\mathcal{A}_1\cdot P=((B)_1^{(2)},\ldots,(B)_{n-1}^{(2)},0).
\end{eqnarray}

For the equation $\mathcal{B}x^{k-1}=0$, where $x=(x_1,x_2,\ldots,x_n)^\top\in\mathbb{C}^n$. By (\ref{equ 6}), it is easy to see that the coefficient of the term $x_n^{n-1}$ in $\mathcal{B}x^{k-1}$ is zero. Hence, the equation $\mathcal{B}x^{k-1}=0$ can be written as
\begin{align*}
\left\{ {\begin{array}{*{20}c}
   f_1(x_1,\ldots,x_{n-1})+x_{n}g_1(x) &= 0 , \\
   \vdots  \\
   f_n(x_1,\ldots,x_{n-1})+x_{n}g_n(x) &= 0,  \\
\end{array}} \right.
\end{align*}
where $f_i(x_1,\ldots,x_{n-1})$ $(i=1,\ldots,n)$ is a homogeneous polynomial with degree $k-1$; $g_i(x)$ $(i=1,\ldots,n)$ is a homogeneous polynomial with degree $k-2$; the degree of $x_{n}$ in $g_i(x)(i=1,\ldots,n)$ is not greater than $k-3$. So the equation $\mathcal{B}x^{k-1}=0$ exists a nonzero solution $x=(0,\ldots,0,x_{n})^\top$, where $x_{n}\neq0$. Hence, ${\rm det}(\mathcal{B})=0$. By Lemma \ref{le2.5}, we get ${\rm det}(\mathcal{B})={\rm det}(\mathcal{A}_1)({\rm det}(P))^{(k-1)^{n-1}}$, so ${\rm det}(\mathcal{A}_1)=0$. It is contradict to that $\mathcal{A}$ is a sign nonsingular tensor. So we get ${\rm r}_i(\mathcal{A}_1)=n$ ($i=1,\ldots,k$). Therefore, ${\rm rank}_\boxplus(\mathcal{A})=(n,\ldots,n)$ for each $\widetilde{\mathcal{A}}\in \mathcal{Q}(\mathcal{A})$. By Lemma \ref{le2.3}, we have ${\rm mr}(\mathcal{A})\geq n$.
\end{proof}
\textbf{Remark of Theorem \ref{thm3.14}}
There exist some sign nonsingular tensors $\mathcal{A}\in\mathbb{R}^{n\times\cdots\times n}$ such that ${\rm mr}(\mathcal{A})> n$. For example
the order $3$ dimension $2$ tensor $\mathcal{A}=(a_{i_1i_2i_3})$, where $a_{111}=2,~a_{112}=0,~a_{121}=0,~a_{122}=3,~a_{211}=0,~a_{212}=3,~a_{221}=0,~a_{222}=0$. For the equation $\mathcal{A}x^{2}=0$ as follows,
\begin{align*}
\left\{ {\begin{array}{*{20}c}
   2x_1^2+3x_2^2 &= 0,  \\
   3x_1x_2 &= 0 . \\
\end{array}} \right.
\end{align*}
It is easy to see that each tensor $\widetilde{\mathcal{A}}\in\mathcal{Q}(\mathcal{A})$ satisfying $\widetilde{\mathcal{A}}x=0$ has an unique solution, that is $x=(0,0)^\top$. So ${\rm det}(\mathcal{A})\neq0$, then $\mathcal{A}$ is a sign nonsingular tensor. According to ${\rm[8]}$,
we can get ${\rm rank}(\widetilde{\mathcal{A}})=3\geq 2$ holds for each tensor $\widetilde{\mathcal{A}}\in \mathcal{Q}(\mathcal{A})$.

For a real matrix $A$, \cite{ck15} showed that the maximum rank of $A$ is equal to the term rank of it, that is ${\rm Mr}(A)=\rho(A)$. In the following result, we give that the maximum rank of a tensor is not less than the term rank of it, that is ${\rm Mr}(\mathcal{A})\geq\rho(\mathcal{A})$.

\begin{thm}\label{thm3.15}
For a tensor $\mathcal{A}\in\mathbb{R}^{n_1\times\cdots\times n_k}$, ${\rm Mr}(\mathcal{A})\geq \rho(\mathcal{A})$.
\end{thm}
\begin{proof} Without loss of generality, suppose that $\mathcal{A}\neq0$. Let $\rho(\mathcal{A})=c$ and $\gamma=[c]$. Then, there exist permutation matrices $P_1,\ldots,P_k$ such that $\mathcal{B}= (P_1,\ldots,P_k)\cdot\mathcal{A}$ has a subtensor $\mathcal{B}[\gamma,\ldots,\gamma]=(b_{i_1\cdots i_k})$ with nonzero diagonal elements, where $i_1,\ldots,i_k\in \gamma$. Take a tensor $\mathcal{B}_1=(\widetilde{b}_{i_1\cdots i_k})\in\mathcal{Q}(\mathcal{B}[\gamma,\ldots,\gamma])$ such that
\[
\widetilde{b}_{i_1\cdots i_k}=\left\{ {\begin{array}{*{20}c}
   \varepsilon b_{i_1\cdots i_k}, & \mbox{if}\ i_1=\cdots=i_k~ (\varepsilon>0), \\
   b_{i_1\cdots i_k}, & \mbox{otherwise}.  \\
\end{array}} \right.
\]
By [2, Proposition 4],
  we get $\varepsilon^{c(k-1)^{c-1}}\prod_{i=1}^c\widetilde{b}_{i\cdots i}^{(k-1)^{c-1}}$ is a term of ${\rm det}(\mathcal{B}_1)$ and the total degree with respect to $\varepsilon$ is not greater than $c(k-1)^{c-1}-2$. Hence, when $\varepsilon$ is large enough, we have
\[
{\rm sgn}({\rm det}(\mathcal{B}_1))={\rm sgn}\left(\varepsilon^{c(k-1)^{c-1}}\prod_{i=1}^c\widetilde{b}_{i\cdots i}^{(k-1)^{c-1}}\right)\neq0.
 \]
By Theorem \ref{thm3.14}, it yields that ${\rm rank}(\mathcal{B}_1)\geq c$, so ${\rm Mr}(\mathcal{B}[\gamma,\ldots,\gamma])\geq c$. By Theorem \ref{thm3.1} and Theorem \ref{thm3.7}, we get ${\rm Mr}(\mathcal{A})={\rm Mr}(\mathcal{B})$ and ${\rm Mr}(\mathcal{B})\geq {\rm Mr}(\mathcal{B}[\gamma,\ldots,\gamma])$, so ${\rm Mr}(\mathcal{A})\geq c$. Therefore, the theorem holds.
\end{proof}

Next is an example of ${\rm Mr}(\mathcal{A})>\rho(\mathcal{A})$.

\begin{example}
For an order $3$ dimension $2$ tensor $\mathcal{A}=(a_{i_1i_2i_3})$, where $a_{111}=2,~a_{112}=0,~a_{121}=0,~a_{122}=1,~a_{211}=1,~a_{212}=-1,~a_{221}=1,~a_{222}=1$.

It is easy to see that $\rho(\mathcal{A})=2$. According to ${\rm [8]}$, we can get ${\rm rank}(\widetilde{\mathcal{A}})=3$ for each $\widetilde{\mathcal{A}}\in \mathcal{Q}(\mathcal{A})$. Therefore, ${\rm Mr}(\mathcal{A})>\rho(\mathcal{A})$.
\end{example}

\cite{ck12} researched the sign nonsingular matrix.
\cite{ck11} studied the left and right inverse of tensors. In the following, we give some results on the sign left and sign right inverse.

\begin{thm}\label{thm3.151}
Let $\mathcal{A}$ be an order $k$ dimension $n$ tensor. If $\mathcal{A}$ has an order $m$ sign left or sign right inverse, then $\mathcal{A}$ is a sign nonsingular tensor.
\end{thm}
\begin{proof}
Suppose that $\mathcal{A}$ has an order $m$ sign left inverse, then for each $\widetilde{\mathcal{A}}\in\mathcal{Q}(\mathcal{A})$ there exists an order $m$ dimension $n$ tensor $\mathcal{P}$ such that $\mathcal{P}\cdot\widetilde{\mathcal{A}}=\mathcal{I}$. By Lemma \ref{le2.5}, it yields that $({\rm det}(\mathcal{P}))^{(k-1)^{n-1}}({\rm det}(\widetilde{\mathcal{A}}))^{(m-1)^n}={\rm det}(\mathcal{I})=1$, so ${\rm det}(\widetilde{\mathcal{A}})\neq0$. Therefore, $\mathcal{A}$ is a sign nonsingular tensor. By the similar proof, we can get that the tensor having an order $m$ sign right inverse is a sign nonsingular tensor.
\end{proof}

Let $\mathcal{A}=(a_{i_1\cdots i_k})\in \mathbb{R}^{n_1\times n_2\times\cdots\times n_2}$ be an order $k$ tensor and let $\mathbb{M}(\mathcal{A})=(m_{ij})\in\mathbb{R}^{n_1\times n_2}$ denote the majorization matrix of $\mathcal{A}$,
where $m_{ij}=a_{ij\cdots j}$ ($i\in [n_1],~j\in [n_2]$) (see \cite{ck28}).

\cite{ck12} gave the sufficiency and necessity of a matrix to be sign nonsingular. In this paper, we show the sufficiency and necessity of a tensor having order $2$ sign left or sign right inverse.

\begin{thm}\label{thm3.17}
Let $\mathcal{A}=(a_{i_1\cdots i_k})\in \mathbb{R}^{n\times \cdots\times n}$ $(k\geq3)$, then $\mathcal{A}=(a_{i_1\cdots i_k})$ has an order 2 sign left inverse if and only if all of the following hold:

{\rm (1)} The elements of $\mathcal{A}$ are all zero expect for $a_{ij\cdots j}$ $(i,~j\in [n])$;

{\rm (2)} The majorization matrix $\mathbb{M}(\mathcal{A})$ is an $SNS$ matrix.
\end{thm}
\begin{proof}
We prove the necessity first. Suppose that $\mathcal{A}$ has an order 2 sign left inverse. It follows from Lemma \ref{le2.6} that for each $\widetilde{\mathcal{A}}\in\mathcal{Q}(\mathcal{A})$ there exists an invertible matrix $\widetilde{P}$ such that $\widetilde{\mathcal{A}}=\widetilde{P}\cdot\mathcal{I}$. By the general tensor product, it yields that (1) holds and $\widetilde{P}=\mathbb{M}(\widetilde{\mathcal{A}})$.
 Hence, each matrix in $\mathcal{Q}(\mathbb{M}(\mathcal{A}))$ is nonsingular. Thus, we get $\mathbb{M}(\mathcal{A})$ is an $SNS$ matrix.

Next, we prove the sufficiency. Since the elements of $\mathcal{A}$ satisfy
\begin{align*}
a_{ij_2\cdots j_k}=\left\{ {\begin{array}{*{20}c}
   m_{ij}, & if\ j_2=\cdots=j_k,~(i,~j_2,\ldots,~j_k\in [n]),\\
   0, & otherwise,  \\
\end{array}} \right.
\end{align*}
then each tensor $\widetilde{\mathcal{A}}\in\mathcal{Q}(\mathcal{A})$ satisfies $\widetilde{\mathcal{A}}=\mathbb{M}(\widetilde{\mathcal{A}})\cdot\mathcal{I}$. Note that $\mathbb{M}(\widetilde{\mathcal{A}})\in\mathcal{Q}(\mathbb{M}(\mathcal{A}))$, then $\mathbb{M}(\mathcal{A})$ is an $SNS$ matrix, it yields that the order 2 left inverse of $\widetilde{\mathcal{A}}$ exists. Thus, we get $\mathcal{A}$ has an order 2 sign left inverse.
\end{proof}

An order $k$ tensor $\mathcal{A}=(a_{i_1\cdots i_k})\in \mathbb{R}^{n\times \cdots\times n}$ is called sign symmetric tensor if ${\rm sgn}(a_{i_1\cdots i_k})={\rm sgn}(a_{\sigma(i_1,\cdots, i_k)})$, where $\sigma(i_1,\cdots,i_k)$ is an arbitrary permutation of the indices $i_1,\cdots ,i_k$.

\begin{thm}\label{thm3.18}
Let $\mathcal{A}=(a_{i_1\cdots i_k})\in \mathbb{R}^{n\times \cdots\times n}$ $(k\geq3)$, then $\mathcal{A}$ has an order 2 sign right inverse if and only if there exist a permutation matrix $P$ and a diagonal matrix $D$ with diagonal element $1$ or $-1$ such that ${\rm sgn}(\mathcal{A})=DP\cdot\mathcal{I}$. Especially, when $k$ is odd, $D=I$.
\end{thm}
\begin{proof}
We prove the sufficiency first. Suppose that there exist a permutation matrix $P$ and a diagonal matrix $D$ with diagonal element $1$ or $-1$ such that ${\rm sgn}(\mathcal{A})=DP\cdot\mathcal{I}$ (when $k$ is odd, $D=I$). For each tensor $\widetilde{\mathcal{A}}\in \mathcal{Q}(\mathcal{A})$, there exists a matrix $\widetilde{D}\in\mathcal{Q}(D)$ such that $\widetilde{\mathcal{A}}=\widetilde{D}P\cdot\mathcal{I}$. Since $D$ is a diagonal matrix with diagonal element $1$ or $-1$ (when $k$ is odd, $D=I$), then there exists a matrix $D_1\in\mathcal{Q}(D)$ such that $D_1^{k-1}=\widetilde{D}$. Let $P_2=P^\top D^{-1}_1$, then
\[
\widetilde{\mathcal{A}}\cdot P_2=\widetilde{D}P\cdot\mathcal{I}\cdot P^\top D^{-1}_1=\mathcal{I},
\]
so each tensor in $\mathcal{Q}(\mathcal{A})$ has an order 2 right inverse.

Next we prove the necessity. Suppose that $\mathcal{A}$ has an order 2 sign right inverse. According to Lemma \ref{le2.6}, we have for each $\widetilde{\mathcal{A}}\in\mathcal{Q}(\mathcal{A})$, there exists an invertible matrix $W\in \mathbb{R}^{n\times n}$ such that
\[
\widetilde{\mathcal{ A}} = \mathcal {I}\cdot W = (I, {W}^\mathrm{T}, \ldots, {W}^\mathrm{T})\cdot\mathcal{ I} = \sum\limits_{i = 1}^n {{e_i} \otimes {w_i} \otimes  \cdots  \otimes {w_i}} ,\]
where $e_i$ is the unit vector whose $i$-th component is $1$, $w_i\neq 0$ denotes the $i$-th row of ${W}$ ($i=1,\ldots,n$).
Unfolding $\widetilde{\mathcal{A}}$ into slices by $1$-th order gives that $\widetilde{\mathcal{A}}=((\widetilde{\mathcal{A}})_1^{(1)},\ldots,(\widetilde{\mathcal{A}})_n^{(1)})$, where each slice $(\widetilde{\mathcal{A}})_i^{(1)}$ is an order $k-1$ tensor which can be written as the rank one form
 $$(\widetilde{\mathcal{A}})_i^{(1)}=w_i\otimes \cdots\otimes w_i,$$
 for each $i\in[n]$. It is easy to see that $(\widetilde{\mathcal{A}})_i^{(1)}$ is a sign  symmetric tensor and ${\rm mr}((\widetilde{\mathcal{A}})_i^{(1)})={\rm Mr}((\widetilde{\mathcal{A}})_i^{(1)})=1~(i=1,\ldots,n)$. Let $p_i$ be the number of the nonzero elements of $w_i$. Obviously, $\rho((\widetilde{\mathcal{A}})_i^{(1)})\geq p_i$. According to Theorem \ref{thm3.15}, we have $1={\rm Mr}((\widetilde{\mathcal{A}})_i^{(1)})\geq p_i\geq1$, so $p_i=1$ ($i=1,\ldots,n$). Let $\widehat{w}_i$ be the only nonzero element of $w_i$ ($i=1,\ldots,n$). So the tensor $(\widetilde{\mathcal{A}})_i^{(1)}$ has only one nonzero element $a_{ij\cdots j}=(\widehat{w}_i)^{k-1}$ ($i=1,\ldots,n$). Especially, when $k$ is odd, ${\rm sgn}(a_{ij\cdots j})=+$. Hence,
\begin{align*}
a_{ij_2\cdots j_k}=\left\{ {\begin{array}{*{20}c}
   m_{ij},& if\ j_2=\cdots=j_k~( i,~j_2,\ldots,j_k\in [n]),  \\
   0, & otherwise . \\
\end{array}} \right.
\end{align*}
So, the majorization matrix $\mathbb{M}(\mathcal{A})=(m_{ij})$. Note that ${\rm sgn}(m_{ij})={\rm sgn}(w_i^j)^{k-1}\neq 0$. Since $W$ is invertible and there is only one nonzero element in each row of $W$, it is easy to see that $\mathbb{M}(\mathcal{A})$ can be permuted into a diagonal matrix.
Therefore, there exist a permutation matrix $P$ and a diagonal matrix $D$ with the diagonal entries $1$ or $-1$ such that $DP{\rm sgn}(\mathbb{M}(\mathcal{A}))=I$.
Thus, we get ${\rm sgn}(\mathcal{A})=DP\cdot\mathcal{I}$.

\end{proof}

\end{document}